\documentclass[11pt, reqno]{amsart}
\usepackage{xcolor}
\usepackage{graphicx}
\usepackage{blindtext}
\usepackage{latexsym}
\usepackage[pagewise]{lineno}
\usepackage[utf8]{inputenc}
\usepackage{exscale}
\usepackage{amsmath}
\usepackage{amssymb}
\usepackage{amsfonts}
\usepackage{mathrsfs}
\usepackage{amsbsy}
\usepackage{relsize}
\usepackage{bbold}
\usepackage[final]{pdfpages}
\usepackage[colorlinks, linkcolor=purple, citecolor=blue, urlcolor=blue, hypertexnames=false]{hyperref}
\usepackage{comment}
\PassOptionsToPackage{reqno}{amsmath}
\usepackage{esint} % For \fint symbol
\usepackage{amssymb} % For \lesssim symbol
\usepackage{stmaryrd}
\usepackage{cite}
\usepackage{geometry}
\newcommand{\R}{\mathbb{R}}

\newcommand{\beq}{\begin{eqnarray}}
\newcommand{\eeq}{\end{eqnarray}}
\newcommand{\bq}{\begin{equation}}
\newcommand{\eq}{\end{equation}}
\newcommand{\beqn}{\begin{eqnarray*}}
\newcommand{\eeqn}{\end{eqnarray*}}
\newcommand{\bex}{\begin{exo}}
\newcommand{\eex}{\end{exo}}
\newcommand{\ben}{\begin{enumerate}}
\newcommand{\een}{\end{enumerate}}

\newtheorem{th1}{{\bf Theorem}}[section]
\newtheorem{thm}[th1]{{\bf Theorem}}
\newtheorem{lem}[th1]{{\bf Lemma}}
\newtheorem{prop}[th1]{{\bf Proposition}}

\newtheorem{rem}[th1]{\bf Remark}

\newtheorem{defi}[th1]{\bf Definition}

\usepackage{lipsum} % To generate some text for the article via \lipsum
\usepackage{tikz,xcolor}

\definecolor{lime}{HTML}{A6CE39}
\DeclareRobustCommand{\orcidicon}{
	\begin{tikzpicture}
	\draw[lime, fill=lime] (0,0) 
	circle [radius=0.16] 
	node[white] {{\fontfamily{qag}\selectfont \tiny ID}};
	\draw[white, fill=white] (-0.0625,0.095) 
	circle [radius=0.007];
	\end{tikzpicture}
	\hspace{-2mm}
}
\foreach \x in {A, B, C}{\expandafter\xdef\csname orcid\x\endcsname{\noexpand\href{https://orcid.org/\csname orcidauthor\x\endcsname}
			{\noexpand\orcidicon}}
}
\usepackage{bm}
\author[A. M. Alqaied \& T. Saanouni]{Alaa Mohammed Alqaied \& Tarek Saanouni$^*$\orcidC{}}
\thanks{* Corresponding author.}
\address[T. Saanouni]{Department of Mathematics, College of Science, Qassim University, Buraydah, Kingdom of Saudi Arabia.}
\email{\sl\textcolor{blue}{t.saanouni@qu.edu.sa}}

\address[A. M. Alqaied]{Department of Mathematics, College of Science, Qassim University, Buraydah, Kingdom of Saudi Arabia.}

\email{\sl\color{blue}{46121523@qu.edu.sa}}

\subjclass[2020]{35Q55}

\subjclass[2010]{35Q55}
\keywords{Biharmonic Inhomogeneous Schr\"odinger equation, nonlinear equations, global existence, blow-up.}

\title[BINLS]{A note on the fourth-order Schr\"odinger equation with spatially growing inhomogeneous source term}

\date{\today}
\begin{document}
\begin{abstract}
This paper studies a non-linear biharmonic Sch\"odinger equation
with an unbounded inhomogeneous term. The main goal is to develop a local theory but also a global theory for small data, in the energy space. Moreover, we develop a local theory in Sobolev spaces with lower regularity. The challenge is to deal with the inhomogeneous unbounded term, which broke down the space translation invariance. In order to handle the inhomogenous term, we use some Strauss type estimates, which require a spherically symmetric assumption.
\end{abstract}
\hrule
\maketitle
%\tableofcontents
\hrule
\vspace{ 1\baselineskip}
\renewcommand{\theequation}{\thesection.\arabic{equation}}
%%%%%%%%%%%%%%%%%%%%%%%%%%%%%%%%%%%%%%%%%%%%%%%%%%%%%%%%%%%%%%%%%%%%%%%%%%%%%%%%%%%%%%%%%
\section{Introduction}
%%%%%%%%%%%%%%%%%%%%%%%%%%%%%%%%%%%%%%%%%%%%%%%%%%%%%%%%%%%%%%%%%%%%%%%%%%%%%%%%%%%%%%%%%%%%
It is the purpose of this note, to investigate the biharmonic inhomogeneous Schr\"odinger equation

\begin{equation}
\left\{
\begin{array}{ll}
i\partial_t v+\Delta^2  v=-\epsilon |x|^{b}|v|^{q-1}v ;\\
v_{|t=0}=v_0.
\label{S}
\end{array}
\right.
\end{equation}

Throughout this work, we employ the bi-Laplacian operator defined as $\Delta^2 := \Delta(\Delta)$, which applies the standard Laplacian operator twice in succession. The system's state is characterized by a complex-valued wave function $v$, dependent on the variables $(t, x) \in \mathbb{R} \times \mathbb{R}^N$, where $N \geq 1$ denotes the spatial dimension. The parameter $\epsilon = \pm 1$ distinguishes between two fundamental regimes: the defocusing case ($\epsilon = +1$) and the focusing case ($\epsilon = -1$). The nonlinearity in our model is represented by the inhomogeneous term $|\cdot|^b$, where the exponent $b > 0$ is a positive real parameter that governs the growth rate of the nonlinear interaction.

The conceptual foundation for this theoretical framework can be traced to the pioneering investigations of \cite{Karpman} and subsequently, \cite{Karpman 1}. Their work was fundamentally motivated by the need to move beyond the standard nonlinear Schr\"odinger equation to accurately model the propagation of high-intensity laser pulses through bulk media exhibiting a Kerr nonlinearity. They recognized that under such conditions, the conventional paraxial approximation was insufficient, as it neglected significant higher-order dispersive effects.

To address this physical limitation, they systematically introduced a fourth-order dispersion term into the governing wave equation, creating a more comprehensive model capable of describing novel nonlinear wave phenomena, including the stabilization of solitons against collapse. This extension provided crucial insights into beam propagation dynamics that were inaccessible within the conventional second-order theory.

In the context of laser-plasma interactions, the source term within this extended framework is interpreted as a nonlinear potential that is self-consistently generated by, and in turn influences, localized perturbations in the electron density. This interpretation of the source term's physical role aligns with and is supported by the earlier theoretical analyses conducted by \cite{brg}.
 
The Shr\"odinger problem \eqref{S} has the two next conserved quantities 

\begin{align}
 \int_{\R^N}|v(t,x)|^2\,dx &:=M(v(t)) = M(v_0);\label{mass}\tag{Mass}\\   
 \int_{\R^N}\Big(|\Delta v(t,x)|^2+\frac{2\epsilon}{1+q}|v(t,x)|^{1+q}|x|^{b}\Big)dx&:=E(v(t)) = E(v_0).\label{nrg}\tag{Energy}
\end{align}

The inhomogeneous Sobolev space $H^2_{rd}$, is said energy space because it allows to use the conservation laws with minimal regularity. If $v$ resolves the equation \eqref{S}, so does the family

\begin{align}
v_\lambda=\lambda^\frac{4+b}{q-1}v(\lambda^4\cdot,\lambda\cdot),\quad\lambda>0.\label{scl}    
\end{align}

The equality $\|v_\lambda\|_{\dot H^{s_c}}=\|v\|_{\dot H^{s_c}}$ gives the so-called critical index

\begin{align}
s_c:=\frac N2-\frac{4+b}{q-1}.    
\end{align}

The mass and $\dot H^s$-critical exponents for the Schr\"odinger equation \eqref{S} are

\begin{align}
    s_c=0\Leftrightarrow q_{m}:=1+\frac{8+2b}N,\quad s_c=s\Leftrightarrow q^{e}_s:=\left\{
\begin{array}{ll}
1+\frac{8+2b}{N-2s},\quad\mbox{if}\quad N>2s;\\
\infty,\quad\mbox{if}\quad  1\leq N\leq2s,
\end{array}
\right.\quad q^e:=q^e_2.
\end{align}

The next inhomogeneous fourth order nonlinear Schr\"odinger problem has attracted many attention last few years.

\begin{align}
    i\partial_t v+\Delta^2  v=\pm |x|^{-b}|v|^{q-1}v,\quad b>0.\label{S'}
\end{align}

The study of this problem was initiated by C.~M.~Guzmán and A.~Pastor~\cite{gp}, who established local well-posedness in $H^2$ for $N \ge 3$, $0 < b < \min\{\tfrac{N}{2}, 4\}$, and for parameters satisfying $\max\{0, \tfrac{2(1-b)}{N}\} < {q-1}$ and $(N-4)(q-1) < 8 - 2b$. Furthermore, they obtained global well-posedness results in the mass-subcritical and mass-critical regimes, i.e., when $\min\{\tfrac{2(1-b)}{N}, 0\} < {q-1} \le \tfrac{2(4-b)}{N}$, under some technical restrictions. Later, Cardoso, Guzmán, and Pastor~\cite{cgp} extended these results by proving local well-posedness in $\dot{H}^s \cap \dot{H}^2$ for $N \ge 5$, $0 < s < 2$, and $0 < b < \min\{\tfrac{N}{2}, 4\}$, under the condition $\max\{\tfrac{8 - 2b}{N}, 1\} < {q-1} < \tfrac{8 - 2b}{N-4}$. The restrictions about low-dimensional cases and a lower bound on the source term exponent were removed in \cite{lz} by use of some bilinear Strichartz-type estimates in Besov spaces. An energy global theory for small data was developed in the inter-critical regime in \cite{gp} and in the energy-critical regime in \cite{gp2}. In the case of the mass-critical nonlinear Schr\"odinger equation, considerable attention has been devoted to the study of solutions that fail to exist globally when the associated energy is negative, as initially analyzed in \cite{cow,Dinh,rbts}. Concerning the scattering behavior, significant progress has been achieved in the analysis of global solutions within the repulsive inter-critical regime. In particular, the first author demonstrated in \cite{st4} that, under the assumption of spherical symmetry of the initial data, global well-posedness and scattering hold. Subsequently, this symmetry restriction was successfully lifted in \cite{cg2}, thereby extending the scattering results to the general, non-radial focusing setting. See \cite{sg2} for the dichotomy of global existence versus finite time blow-up, under the ground state threshold. Moreover, the energy scattering in the defocusing regime was derived by the first author in \cite{st1}. In the energy-critical case, a local theory was developed by the first author in \cite{sg1,sp1}. Subsequently, in a series of works \cite{akr1,akr2,akr3,akr4}, An. et al. carried out a comprehensive analysis of the inhomogeneous biharmonic nonlinear Schr\"odinger equation given by \eqref{S}. In these studies, they rigorously proved both the local and global well-posedness of the problem, together with the standard continuous dependence of the solutions in the Sobolev space $H^{s}$. Their results hold for spatial dimensions $N \in \mathbb{N}$ and regularity indices satisfying $0 \leq s < \min\left\{\frac{2+N}{2}, \frac{3N}{2}\right\}$, under the additional parameter restrictions $0 < b < \min\left\{4,N, \frac{3N}{2}-s, \frac{N}{2}+2-s\right\}$ and $0 < q < q_s^e$. Furthermore, in the critical case $q=q_s^e$, it was demonstrated in these works that the Cauchy problem associated with equation \eqref{S} admits a locally well-posed formulation in the Sobolev space $H^{s}(\mathbb{R}^{N})$, provided certain parameter conditions are fulfilled. Specifically, the authors considered the case when $N \geq 3$, the regularity index satisfies $1 \leq s < \frac{N}{2}$, and the coefficient $b$ lies within the range $0 < b < \min\left\{4, 2 + \frac{N}{2} - s \right\}$, whenever $q$ is an even integer, or alternatively, when $q-1>\lceil s \rceil-2$\footnote{$\lceil \cdot \rceil$ denote the integer part.}.

It is the aim of this note, to develop a local theory, but also a global theory for small datum, about the inhomogeneous bi-harmonic Schr\"odinger problem \eqref{S} in the case of unbounded growing up inhomogeneous term, namely $b>0$. Indeed, in the energy space $H^2_{rd}$, we obtain the local existence and uniqueness of solutions in the energy sub-critical, but also the energy-critical regime with small data. The proof follows a standard fixed point argument coupled with Strichartz estimates, in Proposition \ref{prop2}, and some Strauss-type estimates about the decay of radial functions in Sobolev spaces, in Lemma \ref{sblv}. Then, we prove that this energy local solution extends globally and scatters, for small datum. In the case of non-finite energy, namely a data in $H^1_{rd}$, we develop a local theory. The main challenge is to deal with the inhomogeneous term, which fails to be in any Lebesque space. Compared with the case $b<0$, the situation is very different and the methods used fail in the case $b>0$. Indeed, the property $|x|^{-b}\in L^r(|x|<1)$ for $r<\frac Nb$ and $|x|^{-b}\in L^r(|x|>1)$ for $r>\frac Nb$ is false for $b\leq0$ but also the property $|x|^{-b}\in L^{\frac Nb,\infty}$ fails for $b\leq0$, and finally, the hardy type estimate $\||x|^{-\alpha}v\|_r\lesssim \||\nabla|^\alpha v\|_{L^r}$ is adapted to the case $b<0$. In this work, we use some Strauss type estimates in order to handle the term $|x|^b$ for $b>0$, which require some spherically symmetric assumption. To the authors knowledge, this work is the first one dealing with the inhomogeneous nonlinear Schr\"odinger problem \eqref{S} with $b>0$. 

For simplicity, we let the standard Lebesgue and Sobolev spaces and norms

\begin{align}
    L^r&:=L^r({\R^N}),\quad H^2:=H^2({\R^N}),\\
    \|\cdot\|_r&:=\|\cdot\|_{L^r},\quad\|\cdot\|:=\|\cdot\|_2,\quad \|\cdot\|_{H^2} := \Big(\|\cdot\|^2 + \|\Delta\cdot\|^2\Big)^\frac12.
\end{align}
Take also the radial Sobolev space $H^2_{rd}:=\{f\in H^2,\quad f(x)=f(|x|)\}$ and $T^{\max}>0$ be the lifespan of an eventual solution to \eqref{S}. Let us also define the real numbers

\begin{align}
%B:=\frac{Np-N-\alpha+2b}2,\quad A:=2p-B;\\%\quad\mbox{and}\quad 
q^e:=q^e_2,\quad D:=\frac{Nq-N-2b}4.%,\quad E:=1+q-D,%\quad\mbox{and}\quad 
%\tilde p:=1+\frac{2b+\alpha}N,\quad \tilde q:=1+\frac{2b}N.%,\quad \bar p:=1+\frac{2b+\alpha}{N-4},\quad N>4.
\end{align}

and the operator

\begin{align}
|\nabla|^s:=\big(-\Delta\big)^\frac s2,\quad \langle|\nabla|^s\rangle v:=\big(v,|\nabla|^\tau v\big),\quad |\langle|\nabla|^\tau\rangle v|&:=\big(|v|^2+||\nabla|^\tau v|^2\big)^\frac12.\label{op}
\end{align}

The next sub-section contains the contribution of this manuscript.
%%%%%%%%%%%%%%%%%%%%%%%%%%%%%%%%%%%%%%%%%%%%%%%%%%%%%%%%%%%%%%%%%%%%%%%%%%%%%%%%%%%%%%%%%%%%%%%%%%%%%%%%%%%%
\subsection{Main results}
%%%%%%%%%%%%%%%%%%%%%%%%%%%%%%%%%%%%%%%%%%%%%%%%%%%%%%%%%%%%%%%%%%%%%%%%%%%%%%%%%%%%%%%%%%%%%%%

We start with a local well-posedness result about the inhomogeneous fourth-order Schr\"odinger problem \eqref{S}, in the energy space.

\begin{thm}\label{loc}
Let $N\geq5$, $\epsilon=\pm1$ and $b\geq0$. Then, for any $v_0\in H^2_{rd}$, there is $T>0$ and a local solution to \eqref{S} in the energy space $C\big([0,T],H^2_{rd}\big)$, whenever $1+\frac2N+\frac{2b}{N-1}\leq  q\leq q^e$, such that $\|v_0\|_{H^2}\ll1$ if $q=q^e$, and one of the following holds

\begin{align}
\left\{
\begin{array}{ll}
N&\geq6,\quad b\leq\frac{(N-1)(4+3N)}{3N};\\ 
   N&=5,\quad \frac29\leq b\leq\frac{38}{45}.
   \end{array}
\right.
\label{cnd1}
\end{align}
Moreover,

\begin{enumerate}
    \item  the mass and energy are conserved;
    \item the solution belongs to $L^p_T(W^{2,r}_{rd})$, for any $(p,r)\in\Gamma$, given in definition \ref{adm};
    \item the uniqueness follows in the energy space for $q<q^e$ or $q=q^e$ and small datum;
    \item the uniqueness follows in $L^p_T(W^{2,r})\cap C\big([0,T],H^2_{rd}\big)$;
    \item The solution is global, whenever $q<q_m$ or $q=q_m$ and $\|v_0\|\ll1$ or $\epsilon=1$ and $q<q^e$.
\end{enumerate}
\end{thm}

In view of the results stated in the above theorem, some comments are in order.

%\begin{rems}
\begin{itemize}
\item[$\spadesuit$]
In the proof, we obtain a stronger result. Indeed, the local existence follows for 

\begin{align}
q&\in \bigcup_{s\in\{\frac12,2\}} \Big[1+\frac{2b}{5-2s}+\frac{2}{5}, 1+\frac{2b}{5-2s}+\frac{5}{3}\Big] \bigcup\Big[3+\frac{2b}{5-2s},9+\frac{2b}{N-2s}\Big],\quad N=5;\nonumber\\
 q&\in \bigcup_{s\in\{\frac12,2\}} \Big[1+\frac{2b}{N-2s}+\frac{2}{N}, 1+\frac{2b}{N-2s}+\frac{8}{N-4}\Big],\quad N\geq6.\nonumber
\end{align}

We choose to omit, in Theorem \ref{loc}, the reunion of intervals with gaps.
\item[$\spadesuit$]
There is a gap if we omit the bounds on $b$.
\item[$\spadesuit$]
Compared with the standard case $b=0$, there is a gap $q<1+\frac{2}{N}$.
\item[$\spadesuit$]
In a paper in progress, we consider the low space dimensions.
\item[$\spadesuit$]
The scattering of global solution is treated in a work in progress.
\end{itemize}

%%%%%%%%%%%%%%%%%
Second, we consider the local well-posedness of the inhomogeneous fourth-order Schr\"odinger problem \eqref{S}, with lower regularity.

\begin{thm}\label{loc2}
Let $N\geq3$, $\epsilon=\pm1$, $0\leq b\leq 4(N-1)$ and $1+\frac{2b}{N-1}\leq q\leq q_1^e$, such that $\|v_0\|_{H^1}\ll1$ if $q=q_1^e$. Then, for any $v_0\in H^1_{rd}$, there is $T>0$ and a local solution to \eqref{S} in the space $C\big([0,T],H^1_{rd}\big)$. Moreover,

\begin{enumerate}
    \item  the mass is conserved;
    \item the solution belongs to $L^p_T(W^{1,r}_{rd})$, for any $(p,r)\in\Gamma$, given in definition \ref{adm};
    \item the uniqueness follows in $C\big([0,T],H^1\big)$ for $q<q^e_1$ or $q=q^e_1$ and small datum;
    \item the uniqueness follows in $L^p_T(W^{1,r})\cap C\big([0,T],H^1_{rd}\big)$.
\end{enumerate}
\end{thm}

In light of the results presented in the preceding theorem, it is appropriate to make several remarks and clarifications concerning their interpretation and implications.
%\begin{rems}
\begin{itemize}
\item[$\spadesuit$]
In the proof, we obtain a stronger result. Indeed, the local existence follows for 

\begin{align}
 q&\in \bigcup_{s\in\{\frac12,1\}} \Big[1+\frac{2b}{N-2s}, 1+\frac{2b}{N-2s}+\frac{8}{N-2}\Big],\quad N\geq 3.
\end{align}
We choose to omit, in Theorem \ref{loc2}, the reunion of intervals with gaps.
\item[$\spadesuit$]
If we omit the bound about $b$, some gaps appear in the range of $q$.
\item[$\spadesuit$]
Compared with Theorem \ref{loc}, there is no gap for lower value of the source term exponent.
\end{itemize}

Finally, we consider a small data theory in the energy space.  

\begin{thm}\label{glb}
Let $N\geq5$, $\epsilon=\pm1$ and $b\geq0$. Suppose that $q\geq 2+\frac{2b}{N-1}$ and $q_m< q<q^e$. Let a local solution to \eqref{S} in the energy space $v\in C\big([0,T],H^2_{rd}\big)$. Then, there is $\delta>0$, such that if $\|e^{i\cdot\Delta^2}v_0\|_{\Lambda^{s_\nu}}<\delta$, then $v$ is global and scatters in $H^2$. Moreover,

\begin{align}
    \|v\|_{\Lambda_{s_\nu}}&\leq 2 \|e^{i\cdot\Delta^2}v_0\|_{\Lambda_{s_\nu}},\quad        \|\langle\Delta\rangle v\|_{\Lambda}\leq 2 \|\langle\Delta\rangle e^{i\cdot\Delta^2}v_0\|_{\Lambda}\label{gl1},
\end{align}

where, we denote by $\nu:=q-1-\frac{2b}{N-1}$, $s_{{\nu}}:=\frac N2-\frac{4}{\nu}{>0}$.

\end{thm}

The results stated in the above theorem prompt the following comments.

\begin{itemize}
\item[$\spadesuit$]
Using Strichartz estimates and Sobolev embeddings, we have $\| e^{i\cdot\Delta^2}v_0\|_{\Lambda^{s_\nu}}\lesssim \|v_0\|_{H^2}$. Hence, the previous result applies particularly when $\|v_0\|_{H^2}\ll1$.
\item[$\spadesuit$]
The condition $2+\frac{2b}{N-1}>q^e$ gives the restriction $b\geq\frac{(N-1)(N-12)}{6}$, which is obvious for $N\leq 12$.
\item[$\spadesuit$]
If $N\geq8$, then $\max\{q_m,2+\frac{2b}{N-1}\}=2+\frac{2b}{N-1}$, so the range should be $ 2+\frac{2b}{N-1}\leq q<q^e$.
\end{itemize}

The next sub-section contains some standard tools.
%%%%%%%%%%%%%%%%%%%%%%%%%%%%%%%%%%%%%%%%%%%%%%%%%%%%%%%%%%%%%%%%%%%%%%%%%%%%%%%%%%%%%%%%%%%%%%%%%%%%%%%%%%%%
\subsection{Useful estimates}
%%%%%%%%%%%%%%%%%%%%%%%%%%%%%%%%%%%%%%%%%%%%%%%%%%%%%%%%%%%%%%%%%%%%%%%%%%%%%%%%%%%%%%%%%%%%%%%

 We start with the next Hardy estimate \cite[Theorem 0.1]{majb}.

\begin{lem}
    Let $1<r<\infty$, $0<s<\frac Nr$ and $u\in \dot W^{s,r}(\R^N)$, then
    \begin{align}
        \||\cdot|^{-s}u\|_r\leq C_{N,s,r}\||\nabla|^su\|_{r}.\label{hrd}
    \end{align}
\end{lem}

Sobolev injections \cite{co} will be useful.
\begin{lem}\label{sblv}
For $N\geq2$, hold
\begin{enumerate}
\item[1.]
$H^2\hookrightarrow L^q$ for every $q\in[2,\frac{2N}{N-4}]$ if $N\geq5$ and every $2\leq q<\infty$ if $N\leq4$;
\item[2.]
 $H^2_{rd} \hookrightarrow\hookrightarrow L^q$ is compact for every $q\in(2,\frac{2N}{N-4})$ if $N\geq5$ and every $2<q<\infty$ if $N\leq4$;
\item[3.]
for any $u\in H^1_{rd}$ and any $\frac12\leq s<1$, holds

\begin{equation}\label{frcs}
\sup_{x\in\R^N}|x|^{\frac{N-2s}2}|u(x)|\leq C_{N,s}\|u\|^{1-s}\|\nabla u\|^s.\end{equation}
\item [4.]
for any $\frac12< s<\frac N2$ and any $u\in\dot H^s_{rd}$, holds

\begin{equation}\label{frcs'}
\sup_{x\in\R^N}|x|^{\frac{N-2s}2}|u(x)|\leq C_{N,s}\||\nabla|^s u\|.\end{equation}
%\item
%$H^\alpha \hookrightarrow W^{1,\frac{2N}{N-2(\alpha-1)}}.$
\end{enumerate}
\end{lem}

\begin{defi}\label{adm}
Letting $N\geq1$ and ${s}<2$, a pair $(q,r)$ is said ${s}$- admissible if 
\begin{align}
    \frac4q+{s}=N\Big(\frac12-\frac1r\Big);\label{dm1}
\end{align}

and

\begin{align}\label{dm2}
   \left\{
\begin{array}{ll}
\frac{2N}{N-2{s}}\leq r<\frac{2N}{N-4},\quad\mbox{if}\quad N\geq5;\\
2\leq r<\infty,\quad\mbox{if}\quad 1\leq N\leq4.
\end{array}
\right.
\end{align}

 Let $\Gamma_{s}$ be the set of ${s}$-admissible couples and $\Gamma:=\Gamma_0$. Let also the Stichartz spaces and norms

 \begin{align}
     \Lambda_{s}(I):=\bigcap_{(q,r)\in\Gamma_{s}}L^q(I,L^r),\quad \|\cdot\|_{\Lambda_{s}(I)}:=\sup_{(q,r)\in\Gamma_{s}}\|\cdot\|_{L^q(I,L^r)},\quad\Lambda:=\Lambda_0.
 \end{align}
  
\end{defi}

Recall the so-called Strichartz estimates \cite{bp,guo}.

\begin{prop}\label{prop2}
Let $N \geq 1$, $t_0\in I\subset \R$, $0\leq {s}<2$ and $(q,r)\in\Gamma_{s},(\tilde q,\tilde r)\in\Gamma_{-{s}}$, then,
\begin{enumerate}
%\sup_{(p,r)\in\Gamma}\inf_{(\tilde q,\tilde r)\in\Gamma'}
\item[1.]
$\|v\|_{L^q(I,L^r)}\lesssim\|u(t_0)\|_{\dot H^{s}}+\|i\partial_t v+\Delta^2 v\|_{L^{\tilde q'}(I,L^{\tilde r'})}$;
\item[2.]
$\|\Delta v\|_{L^q(I,L^r)}\lesssim\|\Delta u(t_0)\|+\|i\partial_t v+\Delta^2 v\|_{L^2(I,\dot W^{1,\frac{2N}{2+N}})}, \quad\forall N\geq3$.
\end{enumerate}
\end{prop}
\begin{rem}
The second statement in the above result reveals a gain in regularity for the corresponding solutions.
\end{rem}

The next Gagliardo-Nirenberg type inequality will be useful.

\begin{prop}\label{gag2}
Let $N\geq1$, $b>0$ and $1+\frac{2b}{N-1}< q< q^e$. Thus, there exists a sharp constant $C_{N,q,b}>0$, such that for all $v\in H^2_{rd}$,
%\begin{equation}\label{ineq2}
\begin{align}
    \int_{\R^N}|v|^{1+q}|x|^{b}\,dx\leq C_{N,q,b}\|v\|^{1+q-D}\|\Delta v\|^{D}.\label{gg}
\end{align}%\end{equation}

\end{prop}
\begin{proof}
    Let us write using Lemma \ref{sblv},

    \begin{align}
        \int_{\R^N}|x|^{b}|v|^{1+q}\,dx
        &=\int_{\R^N}\big(|x|^{\frac{N-1}{2}}|v|\big)^\frac{2b}{N-1}|v|^{1+q-\frac{2b}{N-1}}\,dx\nonumber\\
        &\lesssim \big(\|v\|\|\nabla v\|\big)^\frac{b}{N-1}\int_{\R^N}|v|^{1+q-\frac{2b}{N-1}}\,dx\nonumber\\
        &\lesssim \big(\|v\|\|\nabla v\|\big)^\frac{b}{N-1}\big(\|v\|^{1-\frac N2(\frac12-\frac1r)}\|\Delta v\|^{\frac N2\big(\frac12-\frac1{1+q-\frac{2b}{N-1}}\big)}\big)^{1+q-\frac{2b}{N-1}}\label{gg1}.
    \end{align}

    Moreover, thanks to the interpolation estimate \cite{gl},
    
    \begin{align}
    \|\nabla\cdot\|_r^2\lesssim\|\cdot\|_r\|\Delta\|_r,\quad \forall\, 1\leq r<\infty\label{ntr},    
    \end{align}
    
    via \eqref{gg1}, we get

\begin{align}
        \int_{\R^N}|x|^{b}|v|^{1+q}\,dx
        &\lesssim \big(\|v\|\|\nabla v\|\big)^\frac{b}{N-1}\big(\|v\|^{1-\frac N2(\frac12-\frac1r)}\|\Delta v\|^{\frac N2\big(\frac12-\frac1{1+q-\frac{2b}{N-1}}\big)}\big)^{1+q-\frac{2b}{N-1}}\nonumber\\
        &\lesssim \big(\|v\|^\frac32\|\Delta\|^\frac12\big)^\frac{b}{N-1}\big(\|v\|^{1-\frac N2(\frac12-\frac1r)}\|\Delta v\|^{\frac N2\big(\frac12-\frac1{1+q-\frac{2b}{N-1}}\big)}\big)^{1+q-\frac{2b}{N-1}}\nonumber\\
        &\lesssim\|v\|^{1+q-D}\|\Delta v\|^{D}.
    \end{align}
    
\end{proof}

The rest of this paper is organized as follows. Section \ref{sec2} develops a local theory in the energy space. Section \ref{sec3} develops a local theory in $H^1$. Section \ref{sec4} develops a global theory in the energy space. 
%%%%%%%%%%%%%%%%%%%%%%%%%%%%%%%%%%%%%%%%%%%%%%%%%%%%%%%%%%%%%%%%%%%%%%%%%%%%%%%%%%%%%%%%%%%%%%%%%%%%%%%%%%%%%%%%%%%%%%%%%%%%%%%%%%%%%%%%%%%%%%%%%%%%%%%%%%%%%%%%%%%
\section{Energy Local well-posedness}\label{sec2}
%%%%%%%%%%%%%%%%%%%%%%%%%%%%%%%%%%%%%%%%%%%%%%%%%%%%%%%
This section is devoted to prove Theorem \ref{loc} about the existence and uniqueness of solutions to the non-linear Schr\"odinger problem \eqref{S}. 
%%%%%%%%%%%%%%%%%%%%%%%%%%%%%%%%%%%%%%%%%%%%%%%%%%%%%%%%%%%%%%%%%%%%%%%%%%%%%%%%%%%%%%%%%%%%%%%%%%%%%%
\subsection{Existence of solutions}
%%%%%%%%%%%%%%%%%%%%%%%%%%%%%%%%%%%%%%%%%%%%%%%%%%%%%%%%%%%%%%%%%%%%%%%%%%%%%%%%%%%%%%%%%%%%%%%%%%%%%%%%%
This sub-section establishes the local existence of energy solutions. A standard fixed point argument is used. Take $R>0$ to be fixed later and $u,v$ in the space

 \begin{align}
     B_T(R):=\Big\{v: \langle\Delta\rangle v\in C\big([0,T],L^2\big)\displaystyle\bigcap_{(p,r)\in\Gamma} L_T^p(L^{r})\quad\mbox{s.\,t}\quad \sup_{(p,r)\in\Gamma}\|\langle\Delta\rangle v \|_{L_T^p(L^{r})}\leq R\Big\},
 \end{align} 
 
equipped with the complete distance

\begin{align}
    d(u,v):=\|u-v\|_{\Lambda(0,T)}=\sup_{(p,r)\in\Gamma}\|u-v\|_{L_T^p(L^r(\R^N))}.\label{dst}
\end{align}

Take the integral function

\begin{align}
\digamma(v):=e^{i\cdot\Delta^2}v_0-i\int_0^\cdot e^{i(\cdot-\tau)\Delta^2}\Big[|v|^{q-1}|x|^{b}v\Big]\,d\tau.\label{ntg}
\end{align}

Let $s\in\{\frac12,2\}$, using Strichartz estimate for $(p,r)\in\Gamma$, it follows that

\begin{align}
d\big(\digamma(v_1),\digamma(v_2)\big)
&\lesssim \big\||x|^{b}\big(|v_1|^{q-1}v_1-|v_2|^{q-1}v_2\big)\big\|_{L^{p'}_T(L^{r'})}\nonumber\\
&\lesssim\big \||x|^{b}\big(|v_1|^{q-1}+|v_2|^{q-1}\big)|v_1-v_2|\big\|_{L^{p'}_T(L^{r'})}\nonumber\\
&\lesssim\big \|\Big(\big(|x|^{\frac{N-2s}{2}}|v_1|\big)^\frac{2b}{N-2s}|v_1|^{q-1-\frac{2b}{N-2s}}+\big(|x|^{\frac{N-2s}{2}}|v_2|\big)^\frac{2b}{N-2s}|v_2|^{q-1-\frac{2b}{N-2s}}\Big)|v_1-v_2|\big\|_{L^{p'}_T(L^{r'})}.\label{2.1}
\end{align}

% change $|x|^{\frac{N-1}{2}}$ with $|x|^{\frac{N-2s}{2}}$, $s\in\{\frac12,2\}$.\\

Therefore, applying \eqref{frcs} and using \eqref{2.1}, we obtain provided that ${q - 1 - \frac{2b}{N - 2s} \geq 0}$,

\begin{align}
d\big(\digamma(v_1),\digamma(v_2)\big)
&\lesssim\sum_{k=1}^2\|v_k\|_{L^\infty_T(H^s)}^{\frac{2b}{N-2s}}\big \||v_k|^{q-1-\frac{2b}{N-2s}}|v_1-v_2|\big\|_{L^{p'}_T(L^{r'})}\nonumber\\
%&\lesssim R^{\frac{2b}{N-2s}}\sum_{k=1}^2\big \||v_k|^{q-1-\frac{2b}{N-2s}}|v_1-v_2|\big\|_{L^{p'}_T(L^{r'})}\nonumber\\
&\lesssim R^{\frac{2b}{N-2s}}\sum_{k=1}^2\big \|\|v_k\|_{r}^{q-1-\frac{2b}{N-2s}}\|v_1-v_2\|_r\big\|_{L^{p'}(0,T)},\label{2.2}
\end{align}

where, we used H\"older estimate with the equality

\begin{align}
    1-\frac1r=(q-1-\frac{2b}{N-2s})\frac1r+\frac1r,
\end{align}

which reads via the admissibility condition,

\begin{align}
  2\leq r:=q+1-\frac{2b}{N-2s}\leq  \frac{2N}{N-4}.\label{r}
\end{align}

Here and hereafter, we take the convention $a\leq\frac10$ means $a<\infty$. Thus,

\begin{align}
  1+\frac{2b}{N-2s}\leq q\leq 1+\frac{2b}{N-2s}+\frac{8}{N-4}.\label{2.3}
\end{align}

Hence, by \eqref{2.2} and \eqref{2.3} via Sobolev embedding, we write 

\begin{align}
d\big(\digamma(v_1),\digamma(v_2)\big)
&\lesssim R^{\frac{2b}{N-1}}\sum_{k=1}^2\big \|v_k\big\|_{L^\infty_T(H^{2})}^{q-1-\frac{2b}{N-2s}}\|v_1-v_2\|_{L_T^p(L^r)}T^{1-\frac{2}{p}}\nonumber\\
&\lesssim R^{q-1}T^{1-\frac{N(-1+q-\frac{2b}{N-2s})}{4(1+q-\frac{2b}{N-2s})}}d(v_1,v_2)\nonumber\\
&= R^{q-1}T^{-\frac{N-4}{4(1+q-\frac{2b}{N-2s})}\big(q-1-\frac{2b}{N-2s}-\frac8{N-4}\big)}d(v_1,v_2).\label{2.4}
\end{align}

%where $(p,r)\in\Gamma$ via \eqref{r}, gives because of \eqref{2.3'},
%\begin{align}   2\leq p:=\frac{8(1+q-\frac{2b}{N-1})}{q-\frac{2b}{N-1}-1}.\label{2.5}\end{align}

Moreover, taking $v_2=0$, in \eqref{2.4}, we write

\begin{align}
\|\digamma(v_1)\|_{\Lambda(0,T)}
&\leq d\big(\digamma(v_1),\digamma(0)\big)+C\|v_0\|\nonumber\\
%&\lesssim R^{\frac{2b}{N-1}}\sum_{k=1}^2\big \|v_k\big\|_{L^\infty_T(H^2)}^{q-1-\frac{2b}{N-1}}\|v_1-v_2\|_{L_T^p(L^r)}T^{1-\frac{2}{p}}\nonumber\\
%&\lesssim R^{q-1}T^{1-\frac{2}{p}}d(v_1,v_2)\nonumber\\
&\lesssim\|v_0\|+ R^{q}T^{-\frac{N-4}{4(1+q-\frac{2b}{N-2s})}\big(q-1-\frac{2b}{N-2s}-\frac8{N-4}\big)}.\label{2.4'}
\end{align}

Now, using Strichartz estimates, via \eqref{frcs}, we consider the derivative term

\begin{align}
    \|\Delta(\digamma(v_1))\|_{L_T^p(L^{r})}
    &\lesssim \|\Delta v_0\|+\|\nabla\big(|x|^b|v_1|^{q-1}v_1\big)\|_{L_T^2(L^\frac{2N}{2+N})}\nonumber\\
    &\lesssim \|\Delta v_0\|+\||x|^b|v_1|^{q-1}\big(|\nabla v_1|+|x|^{-1}|v_1|\big)\|_{L_T^2(L^\frac{2N}{2+N})}\nonumber\\
    &= \|\Delta v_0\|+\|\big(|x|^\frac{N-2s}{2}|v_1|\big)^\frac{2b}{N-2s}|v_1|^{q-1-\frac{2b}{N-2s}}\big(|\nabla v_1|+|x|^{-1}|v_1|\big)\|_{L_T^2(L^\frac{2N}{2+N})}\nonumber\\
    &\lesssim \|\Delta v_0\|+\|v_1\|_{L_T^\infty(H^s)}^\frac{2b}{N-2s}\||v_1|^{q-1-\frac{2b}{N-2s}}\big(|\nabla v_1|+|x|^{-1}|v_1|\big)\|_{L_T^2(L^\frac{2N}{2+N})}\label{2.5'}.
\end{align}

We discuss several cases.\\

$\bullet$ First case 

\begin{align}
{N\geq5, \quad 1+\frac{2b}{N-2s}+\frac{4}{N-4}\leq  q\leq 1+\frac{2b}{N-2s}+\frac{8}{N-4}}.\label{cs1}
\end{align}

By \eqref{2.5'} via \eqref{hrd}, H\"older estimate and Sobolev embedding, we write

\begin{align}
    \|\Delta(\digamma(v_1))\|_{L_T^p(L^{r})}
    &\lesssim \|\Delta v_0\|+R^\frac{2b}{N-2s}\big\||v_1|^{q-1-\frac{2b}{N-2s}}\big(|\nabla v_1|+|x|^{-1}|v_1|\big)\big\|_{L_T^2(L^\frac{2N}{2+N})}\nonumber\\
    &\lesssim \|\Delta v_0\|+R^\frac{2b}{N-2s}\Big\|\|v_1\|_{1/(\frac1r-\frac2N)}^{q-1-\frac{2b}{N-2s}}\big(\|\nabla v_1\|_{1/(\frac1r-\frac1N)}+\||x|^{-1}|v_1|\|_{1/(\frac1r-\frac1N)}\big)\Big\|_{L^2(0,T)}\nonumber\\
    &\lesssim \|\Delta v_0\|+R^\frac{2b}{N-2s}\Big\|\|\Delta v_1\|_{r}^{q-\frac{2b}{N-2s}}\Big\|_{L^2(0,T)}\label{2.6},
\end{align}

where

\begin{align}
    r&<\frac N2,\quad 2\leq r\leq\frac{2N}{N-4};\label{n5}\\
    \frac{2+N}{2N}&=(q-1-\frac{2b}{N-2s})(\frac{1}{r}-\frac2N)+\frac{1}{r}-\frac1N.
%    \alpha&=q-1-\frac{2b}{N-1}.
\end{align}

Thus, we get ${N\geq5}$ and 

\begin{align}
    \frac{N-4}{2N}\leq\frac1r=\frac2N+\frac{1}{2(q-\frac{2b}{N-2s})}\leq\frac12,\quad \frac1r>\frac2N.\label{2.7}
\end{align}

The estimates in \eqref{2.7} are

\begin{align}
   \frac{N}{N-4}\leq  q-\frac{2b}{N-2s}\leq \frac{N}{N-8}.\label{2.8'}
\end{align}

Clearly, \eqref{2.8'} is satisfied by \eqref{cs1}. Hence, by \eqref{2.6}, via H\"older estimate, we get 

\begin{align}
   \|\Delta(\digamma(v_1))\|_{L_T^p(L^{r})}
    &\lesssim \|\Delta v_0\|+R^\frac{2b}{N-2s}\|\Delta v_1\|_{L^{2(q-\frac{2b}{N-2s})}_T(L^r)}^{q-\frac{2b}{N-2s}}\nonumber\\
    &\lesssim \|\Delta v_0\|+R^\frac{2b}{N-2s}T^{\frac12-\frac{q-\frac{2b}{N-2s}}{p}}\|\Delta v_1\|_{L^{p}_T(L^r)}^{q-\frac{2b}{N-2s}}\nonumber\\
    &\lesssim \|\Delta v_0\|+R^qT^{-\frac{N-4}8\big({q-1-\frac{2b}{N-2s}}-\frac{8}{N-4}\big)}\label{2.10}.
\end{align}

Here, by \eqref{2.7}, we check that 

$$\frac{1}{p}-\frac{1}{2(q-\frac{2b}{N-2s})}=-\frac{4+N}8\big(\frac1{q-\frac{2b}{N-2s}}-\frac{N-4}{N+4}\big)\leq0.$$
%%%%%%%%%%%%%%%%%%%%%%%%%%%%%%%%%5
$\bullet$ Second case 

\begin{align}
{N\geq5, \quad 1+\frac{2b}{N-2s}+\frac{2}{N-4}\leq  q< 1+\frac{2b}{N-2s}+\frac{4}{N-4}}.\label{cs2}
\end{align}

By \eqref{2.5'} via \eqref{hrd}, H\"older estimate and Sobolev embedding, we write

\begin{align}
    \|\Delta(\digamma(v_1))\|_{L_T^p(L^{r})}
    &\lesssim \|\Delta v_0\|+R^\frac{2b}{N-2s}\big\||v_1|^{q-1-\frac{2b}{N-2s}}\big(|\nabla v_1|+|x|^{-1}|v_1|\big)\big\|_{L_T^2(L^\frac{2N}{2+N})}\nonumber\\
    &\lesssim \|\Delta v_0\|+R^\frac{2b}{N-2s}\Big\|\|v_1\|_{2^*}^{q-1-\frac{2b}{N-2s}}\big(\|\nabla v_1\|_{r}+\||x|^{-1}|v_1|\|_{r}\big)\Big\|_{L^2(0,T)}\nonumber\\
    &\lesssim \|\Delta v_0\|+R^\frac{2b}{N-2s}\Big\|\|v_1\|_{H^2}^{q-1-\frac{2b}{N-2s}}\|\langle\Delta v_1\rangle\|_{r}\Big\|_{L^2(0,T)}\label{cs2-2.6},
\end{align}

where, $2^*_\tau:=\frac{2N}{N-2\tau}$, $2^*:=2^*_2$ and

\begin{align}
    r&< N,\quad 2\leq r\leq\frac{2N}{N-4};\\
    \frac{2+N}{2N}&=(q-1-\frac{2b}{N-2s})(\frac{1}{2}-\frac2N)+\frac{1}{r}.
%    \alpha&=q-1-\frac{2b}{N-1}.
\end{align}

Thus, we get ${N\geq5}$ and 

\begin{align}
    \frac{N-4}{2N}\leq\frac1r=\frac{2+N-(N-4)(q-1-\frac{2b}{N-2s})}{2N}\leq\frac12,\quad \frac1r>\frac1N.\label{cs2-2.7}
\end{align}

The estimates in \eqref{cs2-2.7} are 

\begin{align}
   \frac{2}{N-4}\leq  q-1-\frac{2b}{N-2s}\leq \frac{6}{N-4},\quad   q-1-\frac{2b}{N-2s}< \frac{N}{N-4}.\label{cs2-2.8'}
\end{align}

Clearly, \eqref{cs2-2.8'} are satisfied by \eqref{cs2}. Hence, by \eqref{cs2-2.6}, via H\"older estimate, we get 

\begin{align}
   \|\Delta(\digamma(v_1))\|_{L_T^p(L^{r})}
    &\lesssim \|\Delta v_0\|+R^{q-1}\|\langle\Delta\rangle v_1\|_{L^2_T(L^r)}\nonumber\\
    &\lesssim \|\Delta v_0\|+R^qT^{\frac12-\frac1{p}}\nonumber\\
    &\lesssim \|\Delta v_0\|+R^qT^{-\frac{N-4}8\big({q-1-\frac{2b}{N-2s}}-\frac{6}{N-4}\big)}\label{cs2-2.10}.
\end{align}

Here, by \eqref{cs2-2.7}, we check that 

$$\frac{1}{2}-\frac{1}{p}=-\frac{N-4}8\big({q-1-\frac{2b}{N-2s}}-\frac{6}{N-4}\big)>0.$$

%%%%%%%%%%
$\bullet$ Third case 

\begin{align}
{N\geq5, \quad 1+\frac{2b}{N-2s}+\frac{2}{N-2}\leq  q< 1+\frac{2b}{N-2s}+\min\{\frac{2}{N-4},\frac{N}{N-2}\}}.\label{cs3}
\end{align}

By \eqref{2.5'} via \eqref{hrd}, H\"older estimate and Sobolev embedding, we write

\begin{align}
    \|\Delta(\digamma(v_1))\|_{L_T^p(L^{r})}
    &\lesssim \|\Delta v_0\|+R^\frac{2b}{N-2s}\big\||v_1|^{q-1-\frac{2b}{N-2s}}\big(|\nabla v_1|+|x|^{-1}|v_1|\big)\big\|_{L_T^2(L^\frac{2N}{2+N})}\nonumber\\
    &\lesssim \|\Delta v_0\|+R^\frac{2b}{N-2s}\Big\|\|v_1\|_{2_1^*}^{q-1-\frac{2b}{N-2s}}\big(\|\nabla v_1\|_{r}+\||x|^{-1}|v_1|\|_{r}\big)\Big\|_{L^2(0,T)}\nonumber\\
    &\lesssim \|\Delta v_0\|+R^\frac{2b}{N-2s}\Big\|\|v_1\|_{H^2}^{q-1-\frac{2b}{N-2s}}\|\langle\Delta v_1\rangle\|_{r}\Big\|_{L^2(0,T)}\label{cs3-2.6},
\end{align}

where

\begin{align}
    r&< N,\quad 2\leq r\leq\frac{2N}{N-4};\\
    \frac{2+N}{2N}&=(q-1-\frac{2b}{N-2s})(\frac{1}{2}-\frac1N)+\frac{1}{r}.
%    \alpha&=q-1-\frac{2b}{N-1}.
\end{align}

Thus, we get ${N\geq5}$ and 

\begin{align}
    \frac{N-4}{2N}\leq\frac1r=\frac{2+N-(N-2)(q-1-\frac{2b}{N-2s})}{2N}\leq\frac12,\quad \frac1r>\frac1N.\label{cs3-2.7}
\end{align}

The estimates in \eqref{cs3-2.7} are 

\begin{align}
   \frac{2}{N-4}\leq  q-1-\frac{2b}{N-2s}\leq \frac{6}{N-4},\quad   q-1-\frac{2b}{N-2s}< \frac{N}{N-2}.\label{cs3-2.8'}
\end{align}

Clearly, \eqref{cs3-2.8'} are satisfied by \eqref{cs3}. Hence, by \eqref{cs3-2.6}, via H\"older estimate, we get 

\begin{align}
   \|\Delta(\digamma(v_1))\|_{L_T^p(L^{r})}
    &\lesssim \|\Delta v_0\|+R^{q-1}\|\langle\Delta\rangle v_1\|_{L^2_T(L^r)}\nonumber\\
    &\lesssim \|\Delta v_0\|+R^qT^{\frac12-\frac1{p}}\nonumber\\
    &\lesssim \|\Delta v_0\|+R^qT^{-\frac{N-2}8\big({q-1-\frac{2b}{N-2s}}-\frac{6}{N-2}\big)}\label{cs3-2.10}.
\end{align}

Here, by \eqref{cs3-2.7}, we check that 

$$\frac{1}{2}-\frac{1}{p}=-\frac{N-2}8\big({q-1-\frac{2b}{N-2s}}-\frac{6}{N-2}\big)>0.$$

$\bullet$ Fourth case 

\begin{align}
{N\geq5, \quad 1+\frac{2b}{N-2s}+\frac{2}{N}\leq  q< 1+\frac{2b}{N-2s}+\frac{2}{N-2}}.\label{cs4}
\end{align}

By \eqref{2.5'} via \eqref{hrd}, H\"older estimate and Sobolev embedding, we write

\begin{align}
    \|\Delta(\digamma(v_1))\|_{L_T^p(L^{r})}
    &\lesssim \|\Delta v_0\|+R^\frac{2b}{N-2s}\big\||v_1|^{q-1-\frac{2b}{N-2s}}\big(|\nabla v_1|+|x|^{-1}|v_1|\big)\big\|_{L_T^2(L^\frac{2N}{2+N})}\nonumber\\
    &\lesssim \|\Delta v_0\|+R^\frac{2b}{N-2s}\Big\|\|v_1\|^{q-1-\frac{2b}{N-2s}}\big(\|\nabla v_1\|_{r}+\||x|^{-1}|v_1|\|_{r}\big)\Big\|_{L^2(0,T)}\nonumber\\
    &\lesssim \|\Delta v_0\|+R^\frac{2b}{N-2s}\Big\|\|v_1\|_{H^2}^{q-1-\frac{2b}{N-2s}}\|\langle\Delta v_1\rangle\|_{r}\Big\|_{L^2(0,T)}\label{cs4-2.6},
\end{align}

where

\begin{align}
    r&< N,\quad 2\leq r\leq\frac{2N}{N-4};\\
    \frac{2+N}{2N}&=(q-1-\frac{2b}{N-2s})\frac{1}{2}+\frac{1}{r}.
%    \alpha&=q-1-\frac{2b}{N-1}.
\end{align}

Thus, we get ${N\geq5}$ and 

\begin{align}
    \frac{N-4}{2N}\leq\frac1r=\frac{2+N-N(q-1-\frac{2b}{N-2s})}{2N}\leq\frac12,\quad \frac1r>\frac1N.\label{cs4-2.7}
\end{align}

The estimates in \eqref{cs4-2.7} are 

\begin{align}
   \frac2N\leq  q-1-\frac{2b}{N-2s}\leq \frac{6}{N},\quad   q-1-\frac{2b}{N-2s}< 1.\label{cs4-2.8'}
\end{align}

Clearly, \eqref{cs4-2.8'} are satisfied by \eqref{cs4}. Hence, by \eqref{cs4-2.6}, via H\"older estimate, we get 

\begin{align}
   \|\Delta(\digamma(v_1))\|_{L_T^p(L^{r})}
    &\lesssim \|\Delta v_0\|+R^{q-1}\|\langle\Delta\rangle v_1\|_{L^2_T(L^r)}\nonumber\\
    &\lesssim \|\Delta v_0\|+R^qT^{\frac12-\frac1{p}}\nonumber\\
    &\lesssim \|\Delta v_0\|+R^qT^{-\frac{N-2}8\big({q-1-\frac{2b}{N-2s}}-\frac{6}{N-2}\big)}\label{cs4-2.10}.
\end{align}

Here, by \eqref{cs4-2.7}, we check that 

$$\frac{1}{2}-\frac{1}{p}=-\frac{N}8\big({q-1-\frac{2b}{N-2s}}-\frac{6}{N}\big)>0.$$

%%%%
Now, we gather \eqref{2.4}, \eqref{2.4'}, \eqref{2.10},\eqref{cs2-2.10}, \eqref{cs3-2.10} and \eqref{cs4-2.10}, to get for some $\delta>0$,

\begin{align}
d\big(\digamma(v_1),\digamma(v_2)\big)&\leq C R^{q-1}T^{-\frac{N-4}8\big({q-1-\frac{2b}{N-2s}}-\frac{8}{N-4}\big)}d(v_1,v_2);\label{2.14}\\
\|\digamma(v_1)\|_{\Lambda(0,T)}&\leq C\|v_0\|+C R^{q}T^{-\frac{N-4}8\big({q-1-\frac{2b}{N-2s}}-\frac{8}{N-4}\big)}\label{2.16};\\
 \|\Delta (\digamma(v_1))\|_{\Lambda(0,T)}&\leq C \|\Delta v_0\|+CR^qT^{\delta}\label{2.15}.
\end{align}

Letting $R:=4C\|v_0\|_{H^2}$ and if $q<q^e$,

\begin{align}
    0<T<\min\Big\{\Big(\frac1{4C R^{q-1}}\Big)^\frac8{(N-4)\big(1+\frac8{N-4}+\frac{2b}{N-2s}-q\big)},\Big(\frac1{4C R^{q-1}}\Big)^\frac1\delta \Big\}, 
\end{align}

or, $0<R,T\ll1$, if $q=q^e$, it follows that $\digamma$ is a contraction of $B_T(R)$. So, a direct application of a Picard fix point argument yields the existence of a unique solution to \eqref{S} in $B_T(R).$ Now, collecting the previous calculus, we obtain the local existence in the range 

\begin{align}
{N\geq5, \quad, s\in\{\frac12,2\}},\nonumber\\
{1+\frac{2b}{N-2s}+\frac{2}{N}\leq  q\leq 1+\frac{2b}{N-2s}+\frac{8}{N-4}},\quad N\geq6;\label{2.31}\\
{q-1-\frac{2b}{N-2s}\in\Big[\frac{2}{N},\frac{N}{N-2}\Big]\bigcup\Big[\frac{2}{N-4}, \frac{8}{N-4}}\Big],\quad N=5.\label{2.32}
%{N\geq5,\quad q\geq2,\quad    1+\frac{2b}{N-1}\leq q\leq1+\frac{2b}{N-1}+\frac{8}{N-4}.}\label{2.31}\\
%{N\geq5, \quad 1+\frac{2b}{N-1}+\frac{4}{N-4}\leq  q\leq 1+\frac{2b}{N-1}+\frac{8}{N-4}}.\label{2.32}
\end{align}

Thus, the local existence follows for $ N\geq 6,$

\begin{align}
 q&\in \bigcup_{s\in\{\frac12,2\}} \Big[1+\frac{2b}{N-2s}+\frac{2}{N}, 1+\frac{2b}{N-2s}+\frac{8}{N-4}\Big]\nonumber\\
&=\Big[1+\frac{2b}{N-1}+\frac2N, 1+\frac{8+2b}{N-4}\Big],\quad\mbox{if}\quad b\leq\frac{(N-1)(4+3N)}{3N}.\label{2.33}
\end{align}

Moreover, the local existence follows for $N=5,$ $b\leq\frac{38}{45}$ and

\begin{align}
 q&\in \bigcup_{s\in\{\frac12,2\}} \Big[1+\frac{2b}{N-2s}+\frac{2}{N}, 1+\frac{2b}{N-2s}+\frac{N}{N-2}\Big] \bigcup\Big[1+\frac{2b}{N-2s}+\frac{2}{N-4},1+\frac{2b}{N-2s}+ \frac{8}{N-4}\Big]\nonumber\\
&= \Big[1+\frac{2b}{N-1}+\frac{2}{N}, 1+\frac{2b}{N-4}+\frac{N}{N-2}\Big] \bigcup\Big[1+\frac{2b}{N-1}+\frac{2}{N-4},q^e\Big].\label{2.36}
\end{align}

We regroup the local existence rage as $N\geq5$ and 

\begin{align}
    N&\geq6,\quad b\leq\frac{(N-1)(4+3N)}{3N},\quad 1+\frac2N+\frac{2b}{N-1}\leq  q\leq q^e;\label{2.40}\\ 
   N&=5,\quad \frac29\leq b\leq\frac{38}{45},\quad 1+\frac2N+\frac{2b}{N-1}\leq  q\leq q^e.
%   b\geq\frac43(N-1),&\quad q\in \Big[1+\frac{2b}{N-1}, 1+\frac{2b}{N-1}+\frac{8}{N-4}\Big]\bigcup\Big[1+\frac{2b}{N-4}, 1+\frac{2b}{N-4}+\frac{8}{N-4}\Big].\label{2.43}
 %\frac12(N-4)&\leq b\leq\frac12(N-1),\quad \min\big\{1+\frac{2b}{N-1}+\frac{4}{N-4},2\big\}\leq  q\leq q^e;\label{2.41}\\ 
\end{align}

The local existence part of Theorem\ref{loc} is finished.

%%%%%%%%%%%%%%%%%%%%%%%%%%%%%%%%%%%%%%%%%%%%%%%%%%%%%%%%%%%%%%%%%%%%%%%%%%%%%%%%%%%%%%%%%%%%%%%%%%%%%%
\subsection{Uniqueness of energy solutions}
%%%%%%%%%%%%%%%%%%%%%%%%%%%%%%%%%%%%%%%%%%%%%%%%%%%%%%%%%%%%%%%%%%%%%%%%%%%%%%%%%%%%%%%%%%%%%%%%%%%%%%%%%

This sub-section establishes the uniqueness of energy solutions. Let $v_1,v_2\in C\big([0,T],H^2_{rd}\big)$ to be two local solutions to \eqref{S}. Thus, by taking $s=2$ in \eqref{2.4}, we write

\begin{align}
d\big(v_1,v_2\big)
&\leq C \Big(\sum_{k=1}^2\big \|v_k\big\|_{L^\infty_T(H^{2})}^{q-1}\Big)T^{\frac{N-4}{4(1+q-\frac{2b}{N-4})}\big(\frac8{N-4}-[q-1-\frac{2b}{N-4}]\big)}d(v_1,v_2)\nonumber\\
&\leq C \Big(\sum_{k=1}^2\big \|v_k\big\|_{L^\infty_T(H^{2})}^{q-1}\Big)T^{\frac{N-4}{4(1+q-\frac{2b}{N-4})}\big(q^e-q\big)}d(v_1,v_2).\label{3.1}
\end{align}

We discuss two cases: 
\begin{enumerate}
    \item First case $q<q^e$. Taking $0<T\ll1$ in \eqref{3.1}, so that 
    
    \begin{align*}
    C\Big(\sum_{k=1}^2\big \|v_k\big\|_{L^\infty_T(H^{2})}^{q-1}\Big)T^{-\frac{N-4}{4(1+q-\frac{2b}{N-4})}\big(q-q^e\big)}<1,    
    \end{align*}
    
    it follows that $v_1=v_2$.

 \item Second case $q=q^e$. Taking $\|v_0\big\|_{H^{2}}\ll1$ and $0<T\ll1$ in \eqref{3.1}, so that with a continuity argument, $C\Big(\sum_{k=1}^2\big \|v_k\big\|_{L^\infty_T(H^{2})}^{q-1}\Big)<1$, it follows that $v_1=v_2$.

\end{enumerate}

%%%%%%%%%%%%%%%%%%%%%%%%%%%%%%%%%%%%%%%%%%%%%%%%%%%%%%%%%%%%%%%%%%%%%%%%%%%%%%%%%%%%%%%%%%%%%%%%%%%%%%
\subsection{Global energy solutions}
%%%%%%%%%%%%%%%%%%%%%%%%%%%%%%%%%%%%%%%%%%%%%%%%%%%%%%%%%%%%%%%%%%%%%%%%%%%%%%%%%%%%%%%%%%%%%%%%%%%%%%%%%

If $\epsilon=1$ and $q<q^e$, by the conservation laws, we have $\sup_{t\in[0,T^{\max})}\|v(t)\|_{H^2}<\infty$, so $v$ is global. Now, take $\epsilon=-1$ and $q\leq q_m$. Using the Gagliardo-Nirenberg estimate \eqref{gg}, we write

\begin{align}
 E(v_0) 
 &=\int_{\R^N}\Big(|\Delta v(t,x)|^2-\frac{2}{1+q}|v(t,x)|^{1+q}|x|^{b}\Big)dx\nonumber\\
 &\geq\|\Delta v(t)\|^2-C\|v\|^{1+q-D}\|\Delta v(t)\|^D\nonumber\\
 &\geq\|\Delta v(t)\|^2\Big(1-C\|v\|^{1+q-D}\|\Delta v(t)\|^{D-2}\Big)\label{4.1}.
\end{align}

Since $q<q_m\Longleftrightarrow D<2$, hence, if $q<q_m$ or $q=q_m$ and $\|v_0\|\ll1$ we have $\sup_{t\in[0,T^{\max})}\|v(t)\|_{H^2}<\infty$, so $v$ is global.

%%%%%%%%%%%%%%%%%%%%%%%%%%%%%%%%%%%%%%%%%%%%%%%%%%%%%%%%%%%%%%%%%%%%%%%%%%%%%%%%%%%%%%%%%%%%%%%%%%%%%%%%%%%%%%%%%%%%%%%%%%%%%%%%%%%%%%%%%%%%%%%%%%%%%%%%%%%%%%%%%%%
\section{Local well-posedness with low regularity}\label{sec3}
%%%%%%%%%%%%%%%%%%%%%%%%%%%%%%%%%%%%%%%%%%%%%%%%%%%%%%%

This section is devoted to prove Theorem \ref{loc2} about the existence and uniqueness of solutions to the non-linear Schr\"odinger problem \eqref{S}. 
%%%%%%%%%%%%%%%%%%%%%%%%%%%%%%%%%%%%%%%%%%%%%%%%%%%%%%%%%%%%%%%%%%%%%%%%%%%%%%%%%%%%%%%%%%%%%%%%%%%%%%
\subsection{Existence of solutions}
%%%%%%%%%%%%%%%%%%%%%%%%%%%%%%%%%%%%%%%%%%%%%%%%%%%%%%%%%%%%%%%%%%%%%%%%%%%%%%%%%%%%%%%%%%%%%%%%%%%%%%%%%
This sub-section establishes the local existence of solutions. A standard fixed point argument is used. Take $T,R>0$ to be fixed later and $u,v$ in the space

 \begin{align}
     B_T(R):=\Big\{v: \langle\nabla\rangle v\in C\big([0,T],L^2\big)\displaystyle\bigcap_{(p,r)\in\Gamma} L_T^p(L^{r})\quad\mbox{s.\,t}\quad \sup_{(p,r)\in\Gamma}\|\langle\nabla\rangle v \|_{L_T^p(L^{r})}\leq R\Big\},
 \end{align} 
 
equipped with the complete distance \eqref{dst}. We take also the integral function \eqref{ntg} and $s\in\{\frac12,1\}$. Thanks to Strichartz estimate for $(p,r)\in\Gamma$, it follows that 

\begin{align}
d\big(\digamma(v_1),\digamma(v_2)\big)
&\lesssim\sum_{k=1}^2\|v_k\|_{L^\infty_T(H^s)}^{\frac{2b}{N-2s}}\big \||v_k|^{q-1-\frac{2b}{N-2s}}|v_1-v_2|\big\|_{L^{p'}_T(L^{r'})}\nonumber\\
%&\lesssim R^{\frac{2b}{N-2s}}\sum_{k=1}^2\big \||v_k|^{q-1-\frac{2b}{N-2s}}|v_1-v_2|\big\|_{L^{p'}_T(L^{r'})}\nonumber\\
&\lesssim R^{\frac{2b}{N-2s}}\sum_{k=1}^2\big \|\|v_k\|_{1/\frac1r-\frac1N}^{q-1-\frac{2b}{N-2s}}\|v_1-v_2\|_r\big\|_{L^{p'}(0,T)},\label{4-2.2}
\end{align}

 where, we used H\"older estimate and \eqref{frcs}, with the equality

\begin{align}
    1-\frac1r=(q-1-\frac{2b}{N-2s})\big(\frac1r-\frac 1N\big)+\frac1r.
\end{align}

This reads via the admissibility condition, for $\lambda:=q-1-\frac{2b}{N-2s}$,

\begin{align}
  2\leq r:=\frac{N(2+\lambda)}{N+\lambda}\leq \frac{2N}{N-4}\quad\mbox{and}\quad r>N.\label{r'}
\end{align}

The estimate \eqref{r'} reads

\begin{align}
    (N-6)\lambda\leq8.\label{0r}
\end{align}
%%%%%%%%%%%%
Moreover, by \eqref{4-2.2}, we have

\begin{align}
d\big(\digamma(v_1),\digamma(v_2)\big)
&\lesssim R^{\frac{2b}{N-1}}\sum_{k=1}^2\big \|v_k\big\|_{L^p_T(\dot W^{1,r})}^{q-1-\frac{2b}{N-2s}}\|v_1-v_2\|_{L_T^p(L^r)}T^{1-\frac{2+\lambda}{p}}\nonumber\\
%&\lesssim R^{q-1}d(v_1,v_2)T^{1-\frac{2+\lambda}{p}}\nonumber\\
&\lesssim R^{q-1}T^{1-\frac{\lambda(N-2)}8}d(v_1,v_2)\nonumber\\
&\lesssim R^{q-1}T^{-\frac{N-2}8\big(q-1-\frac{2b}{N-2s}-\frac{8}{N-2}\big)}d(v_1,v_2).\label{02.4}
\end{align}
%%%%%%%%%%%%%%

Here, we suppose that $\lambda\geq0$ and 

\begin{align}
    1\geq\frac{2+\lambda}{p}
    &=\frac{2+\lambda}4\frac4{p}\nonumber\\
    &=\frac{2+\lambda}4N(\frac12-\frac1r)\nonumber\\
    &=\frac{2+\lambda}4N(\frac12-\frac{N+\lambda}{N(2+\lambda)})\nonumber\\
    &=\frac{\lambda(N-2)}8\label{02.5}.
\end{align}

So, \eqref{0r} and \eqref{02.5} read

\begin{align}
    0\leq q-1-\frac{2b}{N-2s}\leq \frac8{N-2}.\label{02.6}
\end{align}

Moreover, taking $v_2=0$, in \eqref{02.4}, we write

\begin{align}
\|\digamma(v_1)\|_{\Lambda(0,T)}
&\leq d\big(\digamma(v_1),\digamma(0)\big)+C\|v_0\|\nonumber\\
%&\lesssim R^{\frac{2b}{N-1}}\sum_{k=1}^2\big \|v_k\big\|_{L^\infty_T(H^2)}^{q-1-\frac{2b}{N-1}}\|v_1-v_2\|_{L_T^p(L^r)}T^{1-\frac{2}{p}}\nonumber\\
%&\lesssim R^{q-1}T^{1-\frac{2}{p}}d(v_1,v_2)\nonumber\\
&\lesssim\|v_0\|+ R^{q}T^{-\frac{N-2}8\big(q-1-\frac{2b}{N-2s}-\frac{8}{N-2}\big)}.\label{02.4'}
\end{align}

%%%%%%%%%%%%%%%%%%%%%%%%%%%%%%%%%%%%%%%%%%%%%%%%%%
Now, we will use the identity, 

\begin{align}
    |\nabla\big(|x|^b|v_1|^{q-1}v_1\big)|
    &\lesssim |x|^b|v_1|^{q-1}\big(|x|^{-1}|v_1|+|\nabla v_1|\big).\label{lp'}
\end{align}

%By Young estimate, we have $|x|^{q-2}|y|\lesssim|x|^{q-1}+|y|^{q-1}$.
So, using Strichartz estimates via \eqref{frcs}, we consider the derivative term  

\begin{align}
    \|\nabla(\digamma(v_1))\|_{L_T^p(L^{r})}
    &\lesssim \|\nabla v_0\|+\Big\||x|^b|v_1|^{q-1}\Big(|x|^{-1}|v_1|+|\nabla v_1|\Big)\Big\|_{L_T^{p'}(L^{r'})}\nonumber\\
    &\lesssim \|\nabla v_0\|+\Big\|\big(|x|^\frac{N-2s}{2}|v_1|\big)^\frac{2b}{N-2s}|v_1|^{q-1-\frac{2b}{N-2s}}\Big(|x|^{-1}|v_1|+|\nabla v_1|\Big)\Big\|_{L_T^{p'}(L^{r'})}\nonumber\\
    &\lesssim \|\nabla v_0\|+\|v_1\|_{L_T^\infty(H^s)}^\frac{2b}{N-2s}\Big\||v_1|^{q-1-\frac{2b}{N-2s}}\Big(|x|^{-1}|v_1|+|\nabla v_1|\Big)\Big\|_{L_T^{p'}(L^{r'})}\label{2.17'}.
%    &+\big\||x|^b|v_1|^{q-1}|\nabla v_1|\big\|_{L_T^{p'}(L^{r'})}+\big\||x|^b|\nabla v_1|^q\big\|_{L_T^{p'}(L^{r'})}\label{2.17}.
\end{align}

Moreover, we write also by \eqref{hrd}, since $r'\leq2<N$,

\begin{align}
   \Big\||v_1|^{q-1-\frac{2b}{N-2s}}\Big(|x|^{-1}|v_1|+|\nabla v_1|\Big)\Big\|_{r'}
   &\lesssim\Big\|\Big(|x|^{-1}+|\nabla|\Big)|v_1|^{q-\frac{2b}{N-2s}-1}v_1\Big\|_{r'}\nonumber\\
   &\lesssim\Big\|\nabla v_1 v_1^{q-\frac{2b}{N-2s}-1}\Big\|_{r'}\label{2.17''}
\end{align}

Hence, by H\"older estimate via \eqref{2.17'} and \eqref{2.17''}, it follows that, for ${q\geq 1+\frac{2b}{N-2s}}$, yields

\begin{align}
    \|\nabla(\digamma(v_1))\|_{L_T^p(L^{r})}
    &\lesssim \|\nabla v_0\|+R^\frac{2b}{N-2s}\Big\|\|v_1\|_{r}^{q-1-\frac{2b}{N-2s}}\|\nabla v_1\|_{1/\frac1r-\frac1N}\Big\|_{L^{p'}(0,T)}\nonumber\\
    &\lesssim \|\nabla v_0\|+R^\frac{2b}{N-2s}T^{1-\frac{2+\lambda}{p}}\|v_1\|_{L_T^p(\dot W^{1,r})}^{q-1-\frac{2b}{N-2s}}\| v_1\|_{L_T^p(\dot W^{,r})}\nonumber\\
     &\lesssim \|\nabla v_0\|+R^qT^{-\frac{N-2}8\big(q-1-\frac{2b}{N-2s}-\frac{8}{N-2}\big)}\label{2.20'}.
\end{align}

Finally, we collect \eqref{2.20'}, \eqref{2.17'} and \eqref{02.4} to get

\begin{align}
     d\big(\digamma(v_1),\digamma(v_2)\big)&\leq C R^{q-1}T^{-\frac{N-2}8\big(q-1-\frac{2b}{N-2s}-\frac{8}{N-2}\big)}d(v_1,v_2)\label{02.28'};\\
     %\|\digamma(v_1)\|_{\Lambda(0,T)}&\leq C\|v_0\|+C R^{q}T^{-\frac{N-4}{4(1+q-\frac{2b}{N-2s})}\big(q-1-\frac{2b}{N-2s}-\frac8{N-4}\big)}\label{2.28'};\\
         \|\langle\nabla\rangle (\digamma(v_1))\|_{L_T^p(L^{r})}     &\leq C \|v_0\|_{H^1}+CR^qT^{-\frac{N-2}8\big(q-1-\frac{2b}{N-2s}-\frac{8}{N-2}\big)}\label{02.29'}.
\end{align}

Letting $R:=2C\|v_0\|_{H^1}$ and, if $q<q^e_1$,

\begin{align}
    0<T<\Big(\frac1{2C R^{q-1}}\Big)^{{\frac{-8}{(N-2)\big(q-1-\frac{2b}{N-2s}-\frac{8}{N-2}\big)}}}, \label{2.30'}
\end{align}

and $0<R,T\ll1$, if $q=q_1^e$, it follows that $\digamma$ is a contraction of $B_T(R)$. So, a direct application of a Picard fix point argument yields the existence of a unique solution to \eqref{S} in $B_T(R).$ The local existence is proved. Now, we obtain the local existence in the range 

\begin{align}
{N\geq3, \quad, s\in\{\frac12,1\}},\nonumber\\
{1+\frac{2b}{N-2s}\leq  q\leq 1+\frac{2b}{N-2s}+\frac8{N-2}}.\label{2.31'}
%{N\geq5,\quad q\geq2,\quad    1+\frac{2b}{N-1}\leq q\leq1+\frac{2b}{N-1}+\frac{8}{N-4}.}\label{2.31}\\
%{N\geq5, \quad 1+\frac{2b}{N-1}+\frac{4}{N-4}\leq  q\leq 1+\frac{2b}{N-1}+\frac{8}{N-4}}.\label{2.32}
\end{align}

Thus, the local existence follows for $N\geq 3,$

\begin{align}
 q&\in \bigcup_{s\in\{\frac12,1\}} \Big[1+\frac{2b}{N-2s}, 1+\frac{2b}{N-2s}+\frac{8}{N-2}\Big]\nonumber\\
&=\Big[1+\frac{2b}{N-1}, 1+\frac{8+2b}{N-2}\Big],\quad\mbox{if}\quad b\leq 4(N-1).\label{2.33'}
\end{align}

%%%%%%%%%%%%%%%%%%%%%%%%%%%%%%%%%%%%%%%%%%%%%%%%%%%%%%%%%%%%%%%%%%%%%%%%%%%%%%%%%%%%%%%%%%%%%%%%%%%%%%
\subsection{Uniqueness of solutions}
%%%%%%%%%%%%%%%%%%%%%%%%%%%%%%%%%%%%%%%%%%%%%%%%%%%%%%%%%%%%%%%%%%%%%%%%%%%%%%%%%%%%%%%%%%%%%%%%%%%%%%%%%

This sub-section establishes the uniqueness of energy solutions. Let $v_1,v_2\in C\big([0,T],H^1_{rd}\big)$ to be two local solutions to \eqref{S}. Thus, by taking account of \eqref{02.28'} with $s=1$, we write

\begin{align}
d\big(v_1,v_2\big)
&\leq C \Big(\sum_{k=1}^2\big \|v_k\big\|_{L^\infty_T(H^{1})}^{q-1}\Big)T^{\frac{N-2}8\big(q^e_1-q\big)}d(v_1,v_2).\label{03.1'}
\end{align}

We discuss two cases: 
\begin{enumerate}
    \item First case $q<q^e_1$. Taking $0<T\ll1$ in \eqref{03.1'}, so that 
    
    \begin{align*}
    C\Big(\sum_{k=1}^2\big \|v_k\big\|_{L^\infty_T(H^{1})}^{q-1}\Big)T^{\frac{N-2}8\big(q^e_1-q\big)}<1,    
    \end{align*}
    
    it follows that $v_1=v_2$.

 \item Second case $q=q_1^e$. Taking $\|v_0\big\|_{H^{1}}\ll1$ and $0<T\ll1$ in \eqref{03.1'}, so that with a continuity argument, $C\Big(\displaystyle\sum_{k=1}^2\big \|v_k\big\|_{L^\infty_T(H^{1})}^{q-1}\Big)<1$, it follows that $v_1=v_2$.

\end{enumerate}

%%%%%%%%%%%%%%%%%%%%%%%%%%%%%%%%%%%%%%%%%%%%%%%%%%%%%%%%%%%%%%%%%%%%%%%%%%%%%%%%%%%%%%%%%%%%%%%%%%%%%%
\section{Small data global theory}\label{sec4}
%%%%%%%%%%%%%%%%%%%%%%%%%%%%%%%%%%%%%%%%%%%%%%%%%%%%%%%%%%%%%%%%%%%%%%%%%%%%%%%%%%%%%%%%%%%%%%%%%%%%%%%%%

This section proves Theorem \ref{glb} about the global existence and scatter of solutions to the non-linear Schr\"odinger problem \eqref{S}. We start with some nonlinear estimates.

\begin{lem}\label{nlest}
The next nonlinear estimates hold.
\begin{align}
\|v^{q-1}w|x|^b\|_{\Lambda'^{-s_\nu}}&\lesssim \|v\|_{L^\infty(H^2)}^{\frac{2b}{N-1}}\|v\|_{\Lambda^{s_\nu}}^{q-1-{\frac{2b}{N-1}}}\|w\|_{\Lambda^{s_\nu}};\label{4.0}\\
        \|v^{q-1}w|x|^b\|_{\Lambda'}&\lesssim \|v\|_{L^\infty(H^2)}^{\frac{2b}{N-1}}\|v\|_{\Lambda^{s_\nu}}^{q-1-{\frac{2b}{N-1}}}\|w\|_{\Lambda};\label{4.2}\\
        \|\Delta\big(|v|^q|x|^b\big)\|_{\Lambda'}&\lesssim \|v\|_{L^\infty(H^2)}^{\frac{2b}{N-1}}\|v\|_{\Lambda^{s_\nu}}^{q-1-{\frac{2b}{N-1}}}\|\Delta v\|_{\Lambda}.\label{4.3}
    \end{align}
\end{lem}
\begin{proof}
    We recall that $r:=1+q-\frac{2b}{N-1}:=2+\nu$, $s_\nu:=\frac N2-\frac{4}{\nu}{>0}$ and we denote by

    \begin{align}
        %(p,r)\in\Lambda,\quad (k,r)\in\Lambda_{s_\nu},\quad (m,r)\in\Lambda_{-s_\nu};\label{4.4}\\
       % \frac{1}{\tilde p'}=\frac{q-\theta}{\hat p},\quad \frac1{p'}=\frac{q-1-\theta}{\hat p}+\frac1p;\label{4.5}\\
       p:=\frac{8(2+\nu)}{N\nu},\quad k:=\frac{4\nu(2+\nu)}{8-(N-4)\nu},\quad m:=\frac{4\nu(2+\nu)}{N\nu^2+(N-4)\nu-8}.\label{4.6}\end{align}

A direct calculus checks that

 \begin{align}
        (p,r)\in\Lambda,\quad (k,r)\in\Lambda_{s_\nu},\quad (m,r)\in\Lambda_{-s_\nu}.\label{4.4}
       % \frac{1}{\tilde p'}=\frac{q-\theta}{\hat p},\quad \frac1{p'}=\frac{q-1-\theta}{\hat p}+\frac1p;\label{4.5}\\
       % p:=\frac{8(2+\nu)}{N\nu},\quad k:=\frac{4\nu(2+\nu)}{8-(N-4)\nu},\quad m:=\frac{4\nu(2+\nu)}{N\nu^2+(N-4)\nu-8}.\label{4.6}
    \end{align}
Using H\"older and Strichar estimates via \eqref{frcs}, we get

\begin{align}
    \big\|v^{q-1}w|x|^b\big\|_{L^{m'}(L^{r'})}
    &= \|\big(|x|^\frac{N-1}{2}|v|\big)^\frac{2b}{N-1}|v|^{q-1-\frac{2b}{N-1}}w\|_{L^{m'}(L^{r'})}\nonumber\\
        &\lesssim\|v\|_{L^\infty(H^2)}^\frac{2b}{N-1} \|v^{q-1-\frac{2b}{N-1}}w\|_{L^{m'}(L^{r'})}\nonumber\\
       % &=\|v\|_{L^\infty(H^2)}^\frac{2b}{N-1} \|v\|_{L^{m'({q-\frac{2b}{N-1}})}(L^{r})}^{q-\frac{2b}{N-1}}\nonumber\\
        &=\|v\|_{L^\infty(H^2)}^\frac{2b}{N-1} \|v\|_{L^{k}(L^{r})}^{q-1-\frac{2b}{N-1}}\|w\|_{L^{k}(L^{r})}\label{4.7},
\end{align}

    which proves \eqref{4.0}. Moreover, using the next identity, for ${q\geq2}$,

\begin{align}
    |\Delta\big(|x|^b|v_1|^{q-1}v_1\big)|
    &\lesssim |x|^b|v_1|^{q-1}|x|^{-2}|v_1|+|x|^b\big(|v_1|^{q-1}|\Delta v_1|+|v_1|^{q-2}|\nabla v_1|^2\big)\nonumber\\
    &+|x|^b|v_1|^{q-1}|x|^{-1}|\nabla v_1|,\label{lp}
\end{align}

 we write

  \begin{align}
    \big\|\Delta\big(|v|^q|x|^b\big)\big\|_{L^{p'}(L^{r'})}
    %&\lesssim \||x|^b\big|v|^{q-1}\big(|x|^{-2}|v|+|\Delta v|+|x|^{-1}|\nabla v|\big)\nonumber\\    &+|x|^b|v|^{q-2}|\nabla v|^2\big\|_{L^{p'}(L^{r'})}\nonumber\\
    &\lesssim \|\big(|x|^\frac{N-1}{2}|v|\big)^\frac{2b}{N-1}|v|^{q-1-\frac{2b}{N-1}}\big(|x|^{-2}|v|+|\Delta v|+|x|^{-1}|\nabla v|\big)\nonumber\\
    &+\big(|x|^\frac{N-1}{2}|v|\big)^\frac{2b}{N-1}|v|^{q-2-\frac{2b}{N-1}}|\nabla v|^2\big\|_{L^{p'}(L^{r'})}\nonumber\\
        &\lesssim\|v\|_{L^\infty(H^2)}^\frac{2b}{N-1} \||v|^{q-1-\frac{2b}{N-1}}\big(|x|^{-2}|v|+|\Delta v|+|x|^{-1}|\nabla v|\big)\|_{L^{p'}(L^{r'})}\nonumber\\
        &+\|v\|_{L^\infty(H^2)}^\frac{2b}{N-1} \||v|^{q-2-\frac{2b}{N-1}}|\nabla v|^2\|_{L^{p'}(L^{r'})}\label{4.8},
\end{align}  

Hence, arguing as in \eqref{4.7} and using \eqref{hrd} with \eqref{ntr}, since $r'<\frac N2$ because $N\geq5$ and $r\geq2$, we get for ${q-2-\frac{2b}{N-1}\geq0}$, 

\begin{align}
    \big\|\Delta\big(|v|^q|x|^b\big)\big\|_{L^{p'}(L^{r'})}
         &\lesssim\|v\|_{L^\infty(H^2)}^\frac{2b}{N-1}\Big\|\|\Delta\big(|v|^{q-1-\frac{2b}{N-1}}v\big)\|_{r'}+\||\Delta v||v|^{q-1-\frac{2b}{N-1}}\|_{r'}\Big\|_{L^{p'}(0,\infty)}\nonumber\\
        &+\|v\|_{L^\infty(H^2)}^\frac{2b}{N-1} \big\||v|^{q-2-\frac{2b}{N-1}}|\nabla v|^2\big\|_{L^{p'}(L^{r'})}\nonumber\\
        %&\lesssim\|v\|_{L^\infty(H^2)}^\frac{2b}{N-1} \|v\|_{L^{k}(L^{r})}^{q-1-\frac{2b}{N-1}}\|\big(|x|^{-2}|v|+|\Delta v|+|x|^{-1}|\nabla v|\big)\|_{L^{p}(L^{r})}\nonumber\\        &+\|v\|_{L^\infty(H^2)}^\frac{2b}{N-1} \|v\|_{L^{k}(L^{r})}^{q-1-\frac{2b}{N-1}}\|\Delta v\|_{L^{p}(L^{r})}\nonumber\\
        &\lesssim\|v\|_{L^\infty(H^2)}^\frac{2b}{N-1} \|v\|_{L^{k}(L^{r})}^{q-1-\frac{2b}{N-1}}\|\Delta v\|_{L^{p}(L^{r})}\label{4.9}.
\end{align}  

This proves \eqref{4.3} and \eqref{4.2} follows similarly.
\end{proof}

Now, we establish Theorem \ref{glb}. Take the space

\begin{align}
    \mathcal{X}:=\Big\{v:\quad\|v\|_{\Lambda_{s_\nu}}\leq2\|e^{i\cdot\Delta^2}v_0\|_{\Lambda_{s_\nu}},\quad \|\langle\Delta\rangle v\|_{\Lambda}\leq2c\|v_0\|_{H^2}\Big\},\label{4.10}
\end{align}

equipped with the complete metric

\begin{align}
d(u,v):=\|u-v\|_{\Lambda}+\|u-v\|_{\Lambda_{s_\nu}}.    \label{4.11}
\end{align}

Using Strichartz estimates and Lemma \ref{nlest}, we write for $v_1,v_2\in \mathcal{X}$,

\begin{align}
d\big(\digamma(v_1),\digamma(v_2)\big)
&\lesssim \big\||x|^{b}\big(|v_1|^{q-1}v_1-|v_2|^{q-1}v_2\big)\big\|_{\Lambda'_{-s_\nu}}+\big\||x|^{b}\big(|v_1|^{q-1}v_1-|v_2|^{q-1}v_2\big)\big\|_{\Lambda'}\nonumber\\
&\lesssim \big\||x|^{b}\big(|v_1|^{q-1}+|v_2|^{q-1}\big)\big(v_1-v_2\big)\big\|_{\Lambda'_{-s_\nu}}+\big\||x|^{b}\big(|v_1|^{q-1}+|v_2|^{q-1}\big)\big(v_1-v_2\big)\big\|_{\Lambda'}\nonumber\\
&\lesssim \big(\|v_1\|_{L^\infty(H^2)}^{\frac{2b}{N-1}}+\|v_2\|_{L^\infty(H^2)}^{\frac{2b}{N-1}}\big)\big(\|v_1\|_{\Lambda^{s_\nu}}^{q-1-{\frac{2b}{N-1}}}+\|v_2\|_{\Lambda^{s_\nu}}^{q-1-{\frac{2b}{N-1}}}\big)\Big(\|v_1-v_2\|_{\Lambda^{s_\nu}}+\|v_1-v_2\|_{\Lambda}\Big).\label{4.12'}
\end{align}

Hence, by \eqref{4.12'}, we have

\begin{align}
d\big(\digamma(v_1),\digamma(v_2)\big)
&\leq \big(\|v_1\|_{L^\infty(H^2)}^{\frac{2b}{N-1}}+\|v_2\|_{L^\infty(H^2)}^{\frac{2b}{N-1}}\big)\big(\|v_1\|_{\Lambda^{s_\nu}}^{q-1-{\frac{2b}{N-1}}}+\|v_2\|_{\Lambda^{s_\nu}}^{q-1-{\frac{2b}{N-1}}}\big)d(v_1,v_2)\nonumber\\
&\leq 2c \big(2c\|v_0\|_{H^2}\big)^{\frac{2b}{N-1}}\big(2\|e^{i\cdot\Delta^2}v_0\|_{\Lambda^{s_\nu}}\big)^{q-1-{\frac{2b}{N-1}}}d(v_1,v_2)\nonumber\\
&\leq 2^qc^{1+{\frac{2b}{N-1}}}\|v_0\|_{H^2}^{\frac{2b}{N-1}}\|e^{i\cdot\Delta^2}v_0\|_{\Lambda^{s_\nu}}^{q-1-{\frac{2b}{N-1}}}d(v_1,v_2).\label{4.12}
\end{align}

Moreover, letting $v_2=0$ in \eqref{4.12}, it follows that

\begin{align}
\|\digamma(v_1)\|_{\Lambda_{s_\nu}}
%&\leq c \|v_1\|_{L^\infty(H^2)}^\theta\big(\|v_1\|_{\Lambda^{s_\nu}}^{q-1-\theta}+\|v_2\|_{\Lambda^{s_\nu}}^{q-1-\theta}\big)d(v_1,v_2)\nonumber\\
&\leq \|e^{i\cdot\Delta^2}v_0\|_{\Lambda_{s_\nu}}+2^qc^{1+{\frac{2b}{N-1}}}\|v_0\|_{H^2}^{\frac{2b}{N-1}}\|e^{i\cdot\Delta^2}v_0\|_{\Lambda^{s_\nu}}^{q-{\frac{2b}{N-1}}}.\label{4.13}
\end{align}

Furthermore, by Lemma \ref{nlest}, we have

\begin{align}
\|\langle\Delta\rangle \digamma(v_1)\|_{\Lambda}
&\leq c\|v_0\|_{H^2}+c\|v_1\|_{L^\infty(H^2)}^{\frac{2b}{N-1}}\|v_1\|_{\Lambda^{s_\nu}}^{q-1-{\frac{2b}{N-1}}}\|\langle\Delta\rangle v\|_{\Lambda}\nonumber\\
&\leq c\|v_0\|_{H^2}+c(2c\|v_0\|_{H^2})^{1+{\frac{2b}{N-1}}}\big(2\|e^{i\cdot\Delta^2}v_0\|_{\Lambda^{s_\nu}}\big)^{q-1-{\frac{2b}{N-1}}}\nonumber\\
&\leq c\|v_0\|_{H^2}+2^qc^{2+{\frac{2b}{N-1}}}\|v_0\|_{H^2}^{1+{\frac{2b}{N-1}}}\|e^{i\cdot\Delta^2}v_0\|_{\Lambda^{s_\nu}}^{q-1-{\frac{2b}{N-1}}}.\label{4.14}
\end{align}

Now, the stability of $\mathcal{X}$ under the flow of $\digamma$ and its contraction, read

\begin{align}
  2^qc^{1+{\frac{2b}{N-1}}}\|v_0\|_{H^2}^{\frac{2b}{N-1}}\|e^{i\cdot\Delta^2}v_0\|_{\Lambda^{s_\nu}}^{q-{\frac{2b}{N-1}}}<   \|e^{i\cdot\Delta^2}v_0\|_{\Lambda_{s_\nu}};\label{4.15'}\\
  2^qc^{2+{\frac{2b}{N-1}}}\|v_0\|_{H^2}^{1+{\frac{2b}{N-1}}}\|e^{i\cdot\Delta^2}v_0\|_{\Lambda^{s_\nu}}^{q-1-{\frac{2b}{N-1}}}<c\|v_0\|_{H^2}.\label{4.15}
\end{align}

The estimates \eqref{4.15}-\eqref{4.15'} are equivalent to

\begin{align} 
  \|e^{i\cdot\Delta^2}v_0\|_{\Lambda^{s_\nu}}<\min\Big\{\Big(\frac{1}{2^qc^{1+{\frac{2b}{N-1}}}\|v_0\|_{H^2}^{{\frac{2b}{N-1}}}}\Big)^\frac1{q-1-{\frac{2b}{N-1}}},\Big(\frac{1}{2^qc^{2+{\frac{2b}{N-1}}}\|v_0\|_{H^2}^{{\frac{2b}{N-1}}}}\Big)^\frac1{q-1-{\frac{2b}{N-1}}}\Big\}.\label{4.16}
\end{align}

The proof of the global existence and \eqref{gl1} is finished with a standard Picard fixed point argument. Now, we prove the energy scattering. Using the integral formula \eqref{ntg} via Strichartz estimates and \eqref{4.3}-\eqref{4.4}, we write for $t,t'>0$,

\begin{align}
    \|e^{-it\Delta^2}v(t)-e^{-it'\Delta^2}v(t')\|_{H^2}
    &\lesssim \|\langle\Delta\rangle\big(|x|^{b}|v|^{q-1}v\big)\|_{\Lambda(t,t')}\nonumber\\
    &\lesssim \|v\|_{L^\infty(H^2)}^{\frac{2b}{N-1}}\|v\|_{\Lambda^{s_\nu}(t,t')}^{q-1-{\frac{2b}{N-1}}}\|\langle\Delta\rangle v\|_{\Lambda(t,t')}\nonumber\\
    &\to0,\quad\mbox{as,}\quad t,t'\to\infty.\label{4.17}
\end{align}

Finally, taking $\psi:=\displaystyle\lim_{t\to\infty}e^{-it\Delta^2}v(t)$, in $H^2$, we get the energy scattering $v(t)\to e^{it\Delta^2}\psi$, in $H^2$. The proof of Theorem \ref{gl1} is achieved.

%%%%%%%%%%%%%%%%%%%%%%%%%%%%%%%%%%%%%%%%%%%%%%%%%%%%%%%%%%%%%%%%%%%%%%%%%%
\hrule

\vspace{0.3cm}

\noindent{\bf\large Acknowledgement.}
The Researchers would like to thank the Deanship of Graduate Studies and Scientific Research at Qassim University for financial support (QU-APC-2026).\\

\vspace{0.3cm}

{\noindent{\bf\large Declarations.}}
%On behalf of all authors, the corresponding author states that 
\vspace{0.3cm}

\begin{itemize}
\item
Funding: Not applicable.
\item
Conflict of interest: Not applicable.
\item
Ethical approval: Not applicable.
\item
Informed consent: Not applicable.
\end{itemize}
\vspace{0.3cm}

\hrule

%%%%%%%%%%%%%%%%%%%%%%%%%%%%%%%%%%%%%%%%%%
%%%%%%%%%%%%%%%%%%%%%%%%%%%%%%%%%%%%%%%%%%%%%%%%%%%%%%%%%%%%%%
%\section*{Data availability statement}
%%%%%%%%%%%%%%%%%%%%%%%%%%%%%%%%%%%%%%%%%%%%%%%%%%%%%%%%%%%%%%%%%%%%%%%%%%%%%%%%%%%%%%%%%%%%%%%%%%%%%%

%{\bf Data availability statement}.On behalf of all authors, the corresponding author states that there is no conflict of interest. No data-sets were generated or analyzed during the current study.
%%%%%%%%%%%%%%%%%%%%%%%%%%%%%%%%%%%%%%%%%%%%%%%%%%%%%%%%%%%%%%%

%\end{thebibliography}

\end{document}